\documentclass[twoside, 10pt]{article}
  \usepackage{fullpage}
  \usepackage{pgf}
\usepackage[ansinew]{inputenc}
\usepackage{graphicx}
\usepackage{amsmath}
\usepackage{amsfonts}
\usepackage{amssymb}
\usepackage{amsxtra}
\usepackage{amstext}
\usepackage{latexsym}
\usepackage{dsfont}
\usepackage{stmaryrd}
\usepackage{mathrsfs}

\newtheorem{theorem}{Theorem}

\newtheorem{corollary}[theorem]{Corollary}

\newtheorem{proposition}[theorem]{Proposition}
\newtheorem{remark}[theorem]{Remark}

\newenvironment{proof}[1][\noindent Proof]{\textbf{#1} }{\hfill \ \rule{0.5em}{0.5em}}

\def\K{\mathcal{K}}

\def\bin#1,#2{#1\choose#2}

\def\nn{\mathbb{N}}

\title{\bf Fixed Point Theorems for Set-Valued Mappings on TVS-Cone Metric Spaces}

\author{Ra\'ul Fierro$^{1,2}$
\\[1ex]
{\footnotesize $^{1}$Instituto de Matem\'atica. Pontificia Universidad Cat\'{o}lica de Valpara\'{\i}so.}\\[-0.5ex]
{\footnotesize $^{2}$Instituto de Matem\'aticas. Universidad de Valpara\'{\i}so.}}

\begin{document}
\maketitle




\noindent{\bf Abstract}
In the context of tvs-cone metric spaces, we prove a Bishop-Phelps and a Caristi's type theorem. These results allow us to prove a fixed point theorem for $(\delta, L)$-weak contraction according to a  pseudo Hausdorff metric defined by means of a cone metric.




\section{Introduction}
Huang and Zhang in \cite{HZ07}, introduced the concept of cone metric space as a generalization of metric space. The most relevant of their work is that these authors gave an example of a contraction on a cone metric space, which is not contraction in a standard metric space. This fact makes it clear that the theory of metric spaces are not flexible enough for the fixed point theory, which it has prompted several authors to publish numerous works on fixed point theory for operators defined on cone metric spaces. Most of these are based in cone metrics taking values in a Banach space, and even, some of them suppose this space is normal, in the sense that this space has a base of neighborhood of zero consisting of order-convex subsets. The main aim of this paper is to provide results for set-valued mappings defined on a cone metric space, whose metric takes values in a quite general topological vector space, since it is only assumed this space is $\sigma$-order complete. In \cite{AK11} (see also, \cite{AR10}),  Agarwal and Khamsi proved a version of Caristi's theorem based in a Bishop-Phelps type result for a cone metric taking values in a Banach space. In this paper, we extend this result, which enables us to prove a more general version of Caristi's theorem for cone metric spaces. Natural consequences are deduced from this fact and, as an application, we prove existence of fixed point for an analogous weak contraction of set-valued mapping defined by Berinde and Berinde in \cite{BB07}, which, in our case, is defined according to a pseudo Hausdorff cone metric.

The paper is organized as follows. In Section 2 some preliminary definitions and facts are given, while in Section 3, Bishop-Phelps and Caristi's theorems are proved. Finally, Section 4 is devoted to an application  to set-valued weak contractions defined by means of a cone metric.
\section{Preliminaries}
Let $E$ be a  topological vector space with $\theta$ as zero element and  usual notations for addition and scalar product. A cone is a nonempty closed subset $P$ of $E$ such that $P\cap(-P)=\{\theta\}$ and for each $\lambda\geq0$, $\lambda P+P\subseteq P$.
Given a cone $P$ of $E$, a partial order is defined on $E$ as $x\preceq y$, if and only if, $y-x\in P$. We denote by $x\prec y$ whenever $x\preceq y$ and $x\neq y$. Moreover, the notations $x\ll y$ means that $y-x$ belongs to $\mathrm{int}(P)$, the interior of $P$. As natural, the notations $x\succeq y$, $x\succ y$ and $x\gg y$  mean $y\preceq x$, $y\prec x$ and $y\ll x$, respectively. In what follows, we assume $P$ is a cone of $E$ such that $E$ is a Riesz space, i.e.\ given $x,y\in E$, there exists the greatest lower bound (infimum) of $\{x,y\}$. Additionally, $E$ is assumed $\sigma$-order complete, which means that every decreasing bounded from below sequence has an infimum.  For notations and facts about ordered vector spaces, we refer to \cite{AT07}.

\begin{remark}\label{r1} For each $a,b,c\in E$ such that $a\preceq b \ll c$, we have $a\ll c$.
\end{remark}

A cone metric space is a pair $(X,d)$, where $X$ is a nonempty set and $d:X\times X\to E$ is a function satisfying the following two conditions: i) for all $x,y\in X$, $d(x,y)=\theta$, if and only if, $x=y$, and ii) for all $x,y,z\in X$, $d(x,y)\preceq d(x,z)+d(y,z)$.

In the sequel, $(X,d)$ stands for a cone metric space.

\begin{remark}\label{r2}
Note that for all $x,y\in X$, $d(x,y)\succeq\theta$, and $d(x,y)= d(y,x)$.
\end{remark}

Let $\{x_n\}_{n\in\nn}$ be a sequence in $X$ and $x\in X$. We say $\{x_n\}_{n\in\nn}$ converges to $x$, if and only if, for every $\epsilon\gg\theta$, there exists $N\in\nn$ such that, for any $n\geq\nn$, we have $d(x_n,x)\ll\epsilon$. The sequence $\{x_n\}_{n\in\nn}$  it said to be a  Cauchy sequence, if and only if, for every $\epsilon\gg\theta$, there exists $N\in\nn$ such that, for any $m,n\geq\nn$, we have $d(x_m,x_n)\ll\epsilon$. The cone metric space $(X,d)$ is said to be complete, if and only if, every Cauchy sequence in $X$ converges to some point $x\in X$.
A subset $F$ of $X$ is said to be closed, if for any sequence $\{x_n\}_{n\in\nn}$ in $F$ converging to $x\in E$, we have $x\in F$.

\begin{remark}\label{r3}
 If $X$ is complete and $F\subseteq X$ is closed, then $F$ is complete.
\end{remark}

Let $\varphi:X\to E$ be a function. We say $\varphi$ is lower semicontinuous, if and only if, for any $\alpha\in E$, the set $\{x\in X:\varphi(x)\preceq\alpha\}$ is closed. For this function, a Br{\o}nsted type order $\preceq_\varphi$ is defined on  $X$
as follows:
$$
x\preceq_\varphi y,\mbox{ if and only if, }d(x,y)\preceq \varphi(x)-\varphi(y).
$$
It is easy to see that $\preceq_\varphi$ is in effect an order relation on $X$.

In the sequel, $\mathcal{LS}(X)$ stands for the space of all lower semicontinuous and bounded below functions from $X$ to $E$.

\begin{remark}\label{r4}
The function $\varphi$ defining $\preceq_\varphi$  is non-increasing.
\end{remark}

\section{Bishop-Phelps and  Caristi type theorems}

The following theorem is an extension of the well-known results by Bishop-Phelps lemma  \cite{BP63}.

\begin{theorem}\label{t1}
Suppose $X$ is $d$-complete. Then, for each $\varphi\in\mathcal{LS}(X)$ and $x_0\in X$ there exists a maximal element $x^*\in X$ such that $x_0\preceq_\varphi x^*$.
\end{theorem}
\begin{proof}
For each $x\in X$, let $S(x)=\{y\in X:x\preceq_{\varphi} y\}$, $x_0\in X$  and  $C$ be a chain in $S(x_0)$.  Since
$S(x)=\{y\in X:\varphi(y)+d(x,y)\preceq \varphi(x)\}$, the lower semicontinuity of
$\varphi+d(x,\cdot)$ implies $S(x)$  is a closed set. Let $e\gg\theta$ and, inductively, define an increasing sequence $\{x_n\}_{n\in\nn}$
as
$$
x_{n}\in S(x_{n-1})\cap C\quad\mbox{with}\quad \varphi(x_{n})\prec (1/2n)e+L_n\ll (1/n)e+L_n,
$$
where $x_0$ is given, $A_n=\{\varphi(y): y\in S(x_{n-1})\cap C\}$ and $L_n=\inf(A_n)$.
Due to $\varphi$ is non-increasing and bounded below, $A_n$ is a chain in $P$ and consequently $\{x_n\}_{n\in\nn}$ is well defined. Moreover, for each $n,p\in\nn$, $x_{n+p}\in A_n$ and hence
$$
d(x_{n},x_{n+p})\preceq \varphi(x_n)-\varphi(x_{n+p})\ll (1/n)e.
$$
Thus, $\{x_n\}_{n\in\nn}$ is  a Cauchy
sequence and accordingly, there exists $x^*\in X$ such that this sequence converges to $x^*$. Since for each $n\in\nn$,
$S(x_n)$ is a closed set, we have $x^*\in S(x_n)$  and thus $x_0\preceq x_n\preceq x^*$. Suppose $y\in
X$ satisfies $x^*\preceq_\varphi y$. We have, for each $n\in\nn$, $d(x_n,y)\preceq
\varphi(x_n)-\varphi(y)\prec(1/n)e$ and hence $\lim_{n\to\infty}d(x_n,y)=0$. This fact implies that $x^*=y$  and therefore $x^*\in X$ is a
maximal element satisfying $x_0\preceq_\varphi x^*$. This concludes the proof.
\end{proof}

In the sequel, we denote by $2^X$ the family of all nonempty subsets of $X$ and by $B(X)$ the subfamily of $2^X$ consisting of all nonempty and bounded subsets of $X$. For a set-valued mapping $T:X\to2^X$ and $x\in X$, we usually denote $Tx$ instead of $T(x)$.

Theorem \ref{t1} enables us to state below a generalized  version of Caristi's theorem.
\begin{theorem}\label{t2}
Suppose $X$ is $d$-complete, $T:X\to 2^X$  is a set-valued mapping and $\varphi\in\mathcal{LS}(X)$. The following two propositions hold:
\begin{description}
\item[(\ref{t2}.1)] If for  each $x\in X$, there exists $y\in Tx$ such that $d(x,y)\preceq \varphi(x)-\varphi(y)$, then, there exists $x^*\in X$ such that $x^*\in Tx^*$.
\item[(\ref{t2}.2)] If for  each $x\in X$ and $y\in Tx$, $d(x,y)\preceq \varphi(x)-\varphi(y)$, then, there exists $x^*\in X$ such that $\{x^*\}= Tx^*$.
\end{description}
\end{theorem}
\begin{proof}
From Theorem \ref{t1}, $\preceq_{\varphi}$ has a maximal element $x^*\in X$. Suppose
there exists $y\in Tx^*$ such that $d(x^*,y)\preceq \varphi(x^*)-\varphi(y)$.
I.e.\ $x^*\preceq_\varphi y$. The
maximality of $x^*$ implies $y=x^*$ and hence (\ref{t2}.1) holds.

Since $Tx^*$ is nonempty, (\ref{t2}.1) implies  $\{x^*\}\subseteq Tx^*$. By applying assumption in (\ref{t2}.2) again and the maximality of $x^*$, we have $Tx^*\subseteq \{x^*\}$, which proves  (\ref{t2}.2) and the proof is complete.
\end{proof}

For single-valued mappings the following corollary holds.

\begin{corollary}\label{c1}
Suppose $X$ is $d$-complete. Let  $f:X\to X$ be a
mapping and $\varphi\in\mathcal{LS}(X)$  such that for each $x\in X$, $d(x,f(x))\preceq \varphi(x)-\varphi(f(x))$.
Then, there exists $x^*\in X$ such that
$x^*= f(x^*)$.
\end{corollary}

A cone metric version of the nonconvex minimization theorem according to Takahashi \cite{Ta91} is stated as follows.

\begin{theorem}\label{t3}
Let $\varphi\in\mathcal{LS}(X)$ such that for any $x_0\in X$ satisfying $\inf_{x\in
X}\varphi(x)\prec \varphi(x_0)$, the following condition holds: there exists $x\in X\setminus\{x_0\}$ such that $d(x_0,x)\preceq \varphi(x_0)-\varphi(x)$.
Then, there exists $x^*\in X$ such that $\inf_{y\in X}\varphi(y)=\varphi(x^*)$.
\end{theorem}
\begin{proof}
Suppose for every $z\in X$, $\inf_{y\in X}\varphi(y)\prec\varphi(z)$ and let $x_0\in X$.
From Theorem \ref{t1}, $\preceq_{\varphi}$ has a maximal
element $x^*\in X$ such that $x_0 \preceq_{\varphi} x^*$. Since $\varphi$ is non-increasing,  $\varphi(x^*)\preceq
\varphi(x_0)$ and the assumption implies that there exists $x\in
X\setminus\{x^*\}$ such that $x^*\preceq_{\varphi} x$. From the maximality of
$x^*$ we have $x=x^*$, which is a contradiction. Therefore, there exists
$z\in X$ such that $\inf_{x\in X}\varphi(x)=\varphi(z)$, which completes the proof.
\end{proof}

\section{Contractions}
We define $H: B(X)\times B(X)\to E$ as
$$
    H(A,B)=\sup\{\sup_{x\in A}d(x,B),\sup_{y\in B}d(y,A)\},
$$
where for each $x\in X$ and a nonempty subset $A$ of $X$, $d(x,A)=\inf_{y\in A}d(x,y)$.
\begin{remark}\label{r5}
When $d$ is a standard metric on $X$, $H$ is the Hausdorff metric on $B(X)$. However, in general, $(B(X),H)$ is not a cone metric space.
\end{remark}

An linear operator $L:E\to E$ is said to be positive, if for any $x\in P$ we have $Lx\in P$.
Let $\K_+(E)$ be the set of all  positive, injective and continuous linear operators $\delta$ from $E$ into itself such that, there exists $0\leq t<1$ satisfying $0\preceq \delta x\preceq tx$, for all $x\in P$. Notice that for each $\delta\in\K_+(E)$ and $x\in E$, $|\delta x|\preceq \delta |x|$.

Following Berinde and Berinde in \cite{BB07}, a set-valued mapping
$T : X\to B(X) $ is called a  $(\delta, L)$-weak
contraction, if  there exist an positive linear operator $L:E\to E$ and $\delta \in\K_+(E)$ such that
\begin{equation}\label{C}
H(Tx,Ty) \preceq  \delta d(x,y) +Ld(y, T x), \quad\mbox{ for all } x,y\in X.
\end{equation}
Let $T:X\to B(X)$ be a set-valued mapping. We say $T$ is $H$-continuous at $x\in A$, if for any sequence $\{x_n\}_{n\in\nn}$ in $A$ converging to $x$, $\{H(Tx_n,Tx)\}_{n\in\nn}$ converges to $\theta$ in $E$.
The mapping $T$ is said to  be a contraction,  if there exists $k\in\K_+(E)$ such that for any $x,y\in X $, $H(Tx,Ty)\preceq kd(x,y)$. Notice that $T$ is a contraction, if and only if, there exists $0\leq t<1$ such that  for any $x,y\in X $, $H(Tx,Ty)\preceq td(x,y)$. When $E$ is a Banach space, $t$ can be chosen as the spectral ratio $\rho(k)$ of $k$ and hence in this case, $k$ is a contraction, if and only if, $\rho(k)<1$. Of course, any contraction is a weak contraction. A selector of $T$ is any function $f:X\to X$ such that $f(x)\in Tx$, for all $x\in X$. We say $T$ satisfies condition (S) if, for any  $\epsilon>0$, there exists a selector  $f_\epsilon$ of $T$ such that for each $x\in X$, $d(x,f_\epsilon(x))\preceq (1+\epsilon)d(x,Tx)$.

\begin{remark}\label{r6}
For $x\in X$ and $A,B\in B(X)$, it is defined $s(x,B)$ and $s(A,B)$ as follows:
$$
s(x,B)=\bigcup_{b\in B}\{\epsilon\succ\theta:d(x,b)\preceq\epsilon\}
$$
and
$$
s(A,B)=\bigcap_{a\in A}s(a,B)\cap\bigcap_{b\in B}s(b,A).
$$
Some authors such as \cite{AM13,AzEtAl13,CB11,MeEtAl15,ShEtAl12} define $k$-contraction as a set-valued
mapping $T:X\to B(X)$ satisfying
\begin{equation}\label{e1}
   kd(x,y)\in s(Tx,Ty),\qquad\mbox{for all }x,y\in X.
\end{equation}
This definition  is more restrictive than our definition of contraction  by making $L=0$ in \eqref{C}. Indeed, even though the functional $H$ is not properly a cone metric, it is easy to see that a set-valued mapping satisfying condition \eqref{e1}, it also satisfies our definition of contraction. Furthermore, condition $\theta\in s(a,A)$ implies $a\in A$ for all $a\in X$ and $A\subseteq X$, even though $A$ is not closed. However, it is not possible to conclude that $a\in A$, if $d(a,A)=0$, even though $A$ is closed. Consequently, condition \eqref{e1} is stronger than our definition of contraction.
\end{remark}

Given a set-valued mapping $T:X\to B(X)$, we denote by $\varphi_T$ the mapping from $X$ to $E$ defined as $\varphi_T(x)=d(x,Tx)$.

\begin{proposition}\label{p1} Let $T:X \to B(X)$ be a $H$-continuous set-valued mapping. Then, $\varphi_T\in\mathcal{LS}(X)$.
\end{proposition}
\begin{proof}
Let $u,v\in X$ and $y\in Tv$. Hence,
$$
\begin{array}{ccl}
 d(u,Tu) & \preceq & d(u,v)+d(v,y)+d(y,Tu)\\
   &\preceq & d(u,v)+d(v,y)+H(Tv,Tu).\\
\end{array}
$$

Consequently,
$
  \varphi_T(u)\preceq\varphi_T(v)+d(u,v)+H(Tu,Tv)
$
and from this, the lower semicontinuity of $\varphi_T$ is obtained.
\end{proof}

\begin{corollary}\label{c2}
Let $T:X \to B(X)$ be a contraction. Then, $\varphi_T\in\mathcal{LS}(X)$.
\end{corollary}

\begin{theorem}\label{t4}
Let $L:E\to E$  be a positive linear operator, $\delta\in\K_+(E)$, and $T:X\to B(X)$ be a $(\delta, L)$-weak contraction satisfying condition (S). Suppose $E$ is $d$-complete and $\varphi_T\in\mathcal{LS}(X)$. Then,  there exists $x^*\in X$ such that $x^*\in Tx^*$.
\end{theorem}
\begin{proof}
We have
$$
H(Tx,Ty) \preceq  \delta d(x,y) +Ld(y, T x)
, \quad\mbox{ for all } x,y\in X.
$$
Hence, for each $y\in Tx$, we have $H(Tx,Ty)\leq \delta d(x,y)$. Define $\varphi:X\to E$ as $\varphi(x)=\left(\frac{1}{1+\epsilon}-\delta\right)^{-1}\varphi_T(x)$, where $\epsilon>0$ is chosen in such a way that $\frac{1}{1+\epsilon}>\delta$. From assumption $\varphi\in\mathcal{LS}(X)$ and since $T$ satisfies condition (S), there exists a selector
$f_\epsilon$ or $T$ such that for each $x\in X$, $d(x,f_\epsilon(x))\preceq(1+\epsilon)d(x,Tx)$.
Hence,
$
d(f_\epsilon(x),Tf_\epsilon(x))\preceq H(Tx,Tf_\epsilon(x))\preceq \delta d(x,f_\epsilon(x))
$
and thus
$$
\left(\frac{1}{1+\epsilon}-\delta\right)d(x,f_\epsilon(x))\preceq d(x,Tx)-d(f_\epsilon(x),Tf_\epsilon(x)).
$$
Consequently, for each $x\in X$, $d(x,f_\epsilon(x))\preceq\varphi(x)-\varphi(f_\epsilon(x))$ and it follows from Corollary \ref{c1} that there exists $x^*\in X$ such that $x^*\in Tx^*$, which concludes the proof.
\end{proof}

\begin{corollary}\label{c3}
Suppose $E$ is $d$-complete and let $T:X\to B(X)$ be a contraction satisfying condition (S). Then,  there exists $x^*\in X$ such that $x^*\in Tx^*$.
\end{corollary}
\begin{proof}
It follows from Corollary \ref{c2} and Theorem \ref{t4}.
\end{proof}
\begin{corollary}\label{c4}
Suppose $E$ is $d$-complete and let $f:X\to X$ be a single valued contraction. Then,  there exists $x^*\in X$ such that $x^*= f(x^*)$.
\end{corollary}
\begin{remark}\label{r7}
Since the condition $d(x,Tx)=0$ does not imply, even if $Tx$ is closed, that $x\in Tx$, it is not possible, in the scenario of cone metric spaces, to prove existence of fixed point for weak contractions, as it was done by Berinde and Berinde in \cite{BB07} for set-valued mapping defined on standard metric spaces. Consequently, Corollary \ref{c1} was crucial in the proof of Theorem \ref{t4}.
\end{remark}

Some emblematic and particular cases of standard weak contractions are the Chatterjea \cite{Ch72} and Kannan \cite{Ka69} contractions. Natural extensions of these concepts are obtained for set-valued mappings defined on cone metric spaces. Corollary \ref{c5} below shows that, under usual conditions, these enjoy of existence of fixed points.
[A natural extension for this type of contractions is given in corollary below. \cite{NK10} \cite{HR73} \cite{Re72}]

\begin{corollary}\label{c5}
Suppose $E$ is $d$-complete and let $T:X\to B(X)$ be a set-valued mapping such that $\varphi_T\in\mathcal{LS}(X)$ and at least one of the following three conditions holds:
\begin{description}
\item[(\ref{c5}.1)] $H(Tx,Ty)\preceq\alpha[ d(x,Tx)+d(y,Ty)]$\quad(Kannan condition) and
\item[(\ref{c5}.2)] $H(Tx,Ty)\preceq \alpha[d(x,Ty)+d(y,Tx)]$\quad(Chatterjea condition),
\end{description}
where $\alpha:E\to E$ is a linear operator satisfying $2\alpha\in\mathcal{K}_+(E)$.

Then, there exists $x\in X$ such that $x\in T(x)$.
\end{corollary}

\noindent\textbf{Acknowledgements} This research was partially supported by Chilean Council for Scientific and Technological Research, under
grant FONDECYT 1120879.

\providecommand{\bysame}{\leavevmode\hbox to 3em%
{\hrulefill}\thinspace}


\begin{thebibliography}{10}

\bibitem{AK11}
R.P. Agarwal and M.A. Khamsi.
\newblock Extension of {C}aristi's fixed point theorem to vector valued metric
  spaces.
\newblock {\em Nonlinear Analysis}, Vol. 74:141--145, 2011.

\bibitem{AT07}
C.D Aliprantis and R.~Tourky.
\newblock {\em Cones and Duality}.
\newblock American Mathematical Society, Providence, Rhode Island, 2007.

\bibitem{AR10}
I.~Altun and V.~Rako\u{c}evi\'{c}.
\newblock Ordered cone metric spaces and fixed point results.
\newblock {\em Computers and Mathematics with Applications}, Vol.
  60:1145--1151, 2010.

\bibitem{AM13}
A.~Azam and N.~Mehmood.
\newblock Multivalued fixed point theorems in cone tvs-cone metric spaces.
\newblock {\em Fixed Point Theory and Applications}, 2013:184, 2013.

\bibitem{AzEtAl13}
A.~Azam, N.~Mehmood, J.~Ahmad, and S.~Radenovi\'c.
\newblock Multivalued fixed point theorems in cone $b$-metric spaces.
\newblock {\em Journal of Inequalities and Applications}, 2013:582, 2013.

\bibitem{BB07}
M.~Berinde and V.~Berinde.
\newblock On a general class of multi-valued weakly {P}icard mappings.
\newblock {\em Journal of Mathematical Analysis and Applications}, Vol.
  326(2):772--782, 2007.

\bibitem{BP63}
E.~Bishop and R.R. Phelps.
\newblock The support functionals of a convex set.
\newblock {\em Proceedings of Symposia in Pure Mathematics VII, Convexity,
  American Mathematical Society}, pages 27--36, 1963.

\bibitem{Ch72}
S.K. Chatterjea.
\newblock Fixed-point theorems.
\newblock {\em C. R. Acad. Bulgare Sci.}, Vol. 25:727-730, 1972.

\bibitem{CB11}
S.H. Cho and J.S. Bae.
\newblock Fixed point theorems for multivalued maps in cone metric spaces.
\newblock {\em Fixed Point Theory and Applications}, page~87, 2011.

\bibitem{HZ07}
L.G. Huang and X.~Zhang.
\newblock Cone metric spaces and fixed point theorems of contractive mappings.
\newblock {\em Journal of Mathematical Analysis and Applications}, Vol.
  332:1468-1476, 2007.

\bibitem{Ka69}
R.~Kannan.
\newblock Some results on fixed points ii.
\newblock {\em Fundamenta Mathematicae}, Vol. 76(4):405--408, 1969.

\bibitem{MeEtAl15}
N.~Mehmood, A.~Azam, and L.D.R. Ko\u{c}inac.
\newblock Multivalued fixed point results in cone metric spaces.
\newblock {\em Topology and its Applicatons}, Vol. 179:156--170, 2015.

\bibitem{NK10}
K.~Neammanee and A.~Kaewkhao.
\newblock Fixed point theorems of multi-valued {Z}amfirescu mapping.
\newblock {\em Journal of Mathematucal Research}, Vol. 2(2):150--156, 2010.

\bibitem{Re72}
S.~Reich.
\newblock Fixed points in locally convex spaces.
\newblock {\em Mathematische Zeitschrift}, Vol. 125(1):17--31, 1972.

\bibitem{ShEtAl12}
W.~Shatanawi, V.C. Raji\'c, and A.~Al-Rawashdeh.
\newblock {M}izoguchi-{T}akahashi-type theorems in tvs-cone metric spaces.
\newblock {\em Fixed Point Theory and Applications}, Vol. 2012:106, 2012.

\bibitem{Ta91}
W.~Takahashi.
\newblock Existence theorems generalizing fixed point theorems for multivalued
  mappings.
\newblock In M.A. Th\'era and J.B. Baillon, editors, {\em Fixed Point Theory
  and Applications, in: Pitman Research Notes in Mathematics Series, vol.
  252.}, pages 397--406. Longman Sci. Tech., Harlow, 1991.

\end{thebibliography}

\end{document}